\documentclass[11pt]{article}
\usepackage{amsmath,amssymb,amsthm}
\usepackage[margin=1.1in]{geometry}
\usepackage[hidelinks]{hyperref}
\hypersetup{pdftitle={Parity families and signed spectra: averaging identities and a strict separation on signed circulants}, pdfauthor={Vaibhav Suvagiya}, pdfsubject={Spectral optimization over edge signings}, pdfkeywords={Signed graphs, Spectral radius, Switching, Circulant graphs, Non-backtracking walks, Ihara zeta function}}

\newtheorem{lemma}{Lemma}
\newtheorem{proposition}[lemma]{Proposition}
\newtheorem{theorem}[lemma]{Theorem}
\newtheorem{corollary}[lemma]{Corollary}

\theoremstyle{remark}
\newtheorem{remark}[lemma]{Remark}
\newtheorem{question}[lemma]{Question}

\newcommand{\F}{\mathbb{F}}

\newcommand{\Ze}{Z_{\mathrm{ev}}}
\newcommand{\rhomax}{\rho}
\DeclareMathOperator{\tr}{tr}

\title{Parity families and signed spectra: averaging identities and a strict separation on signed circulants}
\author{Vaibhav Suvagiya\\
\small Sardar Vallabhbhai National Institute of Technology, Surat, India\\
\small\href{mailto:vaibhavsuvagiya17@gmail.com}{\texttt{vaibhavsuvagiya17@gmail.com}}}
\date{}

\begin{document}
\maketitle

\begin{abstract}
For every even $n\ge10$, we determine the minimum spectral radius among
signings of the circulant $C_n(1,2)$ in which every quadrilateral is
unbalanced. These signings form four switching classes. If $n\ge32$ is
divisible by eight, a period-eight signing has spectral radius
$r_*\approx2.79360$, strictly below the constrained minimum. Reversing the
signs of the three edges across the cyclic boundary reduces the radius
further. The exact unrestricted minimum remains open. For a general affine
family of edge signings, we express the mean even spectral moments as signed sums of
closed-walk counts indexed by edge parity. An averaged Ihara identity
gives the corresponding decomposition for non-backtracking walks.
We also bound totally even walks at finite scales under local rank and
growth hypotheses, and derive a near-Ramanujan existence result for
bounded-codimension families by conditioning a known random-signing
theorem on bicycle-free graph sequences.
\end{abstract}


\section{Introduction}

A signing of a graph assigns a sign to each edge. Its signed adjacency
matrix $A_\sigma$ records these signs, and its spectral radius
$\rho(A_\sigma)$ is the largest absolute eigenvalue. Switching the signs
at vertices preserves the spectrum and the sign product on each cycle.
It is therefore natural to try to reduce the spectral radius by fixing
the sign products on short cycles. How much can this local choice cost?

The circulant $C_n(1,2)$ gives an exact comparison. Its vertices are the
residues modulo $n$, with edges of steps one and two. For even $n\ge10$,
requiring each quadrilateral to have negative sign product leaves four
switching classes. Their minimum spectral radius is
\[
 \rho_-(n)=2\sqrt{\cos^2(\pi/n)+\cos^2(2\pi/n)}.
\]
On $n=8m\ge32$ vertices the period-eight signing in
Theorem~\ref{thm:counterexample} has radius $r_*<\rho_-(n)$.
Its positive sign product around the step-one Hamilton cycle allows a
constant cell Fourier mode that attains $r_*$. Reversing the three edge
signs across the cyclic boundary changes that product to $-1$ and
excludes the extremal mode. Theorem~\ref{thm:antiperiodic} proves a further
strict decrease. At $n=32$ we also give an independent integer
certificate for a spectral radius below $2.79$.

The broader signing problem arises in graph lifts. Bilu and Linial
\cite{bilulinial06} asked whether a $d$-regular graph always admits a
signing of radius at most $2\sqrt{d-1}$. In a $2$-lift the eigenvalues of
the base are inherited and the signed adjacency spectrum supplies the
new eigenvalues. A two-sided signing bound at this threshold therefore
gives a Ramanujan lift whenever the base is Ramanujan. The one-sided
theorem of Marcus, Spielman and Srivastava~\cite{mss15} resolves the
bipartite case. Section~\ref{sec:related} places the unrestricted and
bicycle-free bounds in relation to the constrained problem.

For a general consistent system of cycle-sign constraints, the mean of
an even trace retains only walks whose edge parity belongs to the
constraint span, with a prescribed sign for each surviving parity
class (Proposition~\ref{prop:master}). The mean logarithmic Ihara
determinant admits a corresponding prime-cycle decomposition
(Proposition~\ref{prop:Lident}). The hypercube gives an example where
quadrilateral constraints determine the spectral radius exactly
(Proposition~\ref{prop:cube}).

Two results delimit the counting approach. Proposition~\ref{prop:upper}
bounds the number of totally even walks at a specified finite scale
under a local rank bound and a growth condition. A clique can instead
support many such walks (Proposition~\ref{prop:trap}).
Proposition~\ref{prop:bf-conditioning} gives a separate transfer of a
random-signing theorem to families of bounded codimension. Neither
estimate resolves cancellation among parity classes on cycle-dense graphs.
The computational checks in Section~\ref{sec:experiments} distinguish
exact algebraic calculations from spectral estimates on finite samples.

\section{Related work and positioning}\label{sec:related}

For a uniform edge signing, averaging the trace cancels each closed walk that
uses an edge an odd number of times. The surviving walks give the spectral
moment formula of Chen, van Dam, and Bu~\cite{chenvandambu24}. Prescribing
cycle signs changes which walks survive. Proposition~\ref{prop:master} retains
exactly those walks whose edge-parity vector belongs to the constraint span
$W$, and weights each surviving class by its prescribed sign.

Li, Hou, and Wang express the zeta function of a single signed graph as a
determinant on directed edges~\cite{lihouwang24}. Luo and Roy recover weighted
circuit counts in individual homology classes by Fourier inversion, including
for finite quotients~\cite[Thm.~3.24(c,d)]{luoroy24}. For edge signings, the
relevant quotient records edge parities modulo two. Varying a signing subject
to fixed cycle signs ranges over a coset of the character subgroup trivial on
$W$. Proposition~\ref{prop:Lident} uses this restriction to group the
prime-cycle contributions to the mean of $\log\det(I-uB_\sigma)$.

Two questions about signings lead to different bounds. On an arbitrary
$d$-regular graph, Ravichandran and Srivastava~\cite{rs21} obtained a signing
of spectral radius at most $2\sqrt{2(d-1)}$. Huang~\cite{huang26} gives a
second proof, and Lin and Zhou~\cite{linzhou26} improve the coefficient of
$\sqrt{d-1}$ to $(3+\sqrt5)/2$. Xu's cubic counterexample~\cite{xu26counter}
rules out coefficient $2$ for arbitrary regular bases and leaves the question
for Ramanujan bases open. On graphs that are bicycle-free to radius
$r\gg(\log\log n)^2$, Mohanty, O'Donnell, and Paredes~\cite[Thm.~1.2]{mop20}
show that a random signing satisfies the near-Ramanujan bound with high
probability. Proposition~\ref{prop:bf-conditioning} uses the positive density
of bounded-codimension families to transfer this result to prescribed cycle
signs.

The author's first parity-family preprint~\cite{prior26} introduced the two
averaging identities. A separate preprint~\cite{companion} determined the four
switching classes on $C_n(1,2)$ and their spectral radii. The author withdrew
the latter preprint when the two works were merged. The consolidated
version~\cite{consolidated26} included the period-$8$ separation of constrained
and unrestricted signings. Theorem~\ref{thm:antiperiodic} treats the opposite Hamilton-cycle sign and
disproves the global-minimum conjecture in that version for every
$n=8m\ge32$.

\section{Setup}

$G=(V,E)$ is a simple connected graph, $|V|=n$, $|E|=m$, adjacency matrix
$A$. A signing is $\sigma\in\{\pm1\}^E$, encoded by the bit vector
$s\in\F_2^E$ with $\sigma_e=(-1)^{s_e}$; the signed adjacency matrix
$A_\sigma$ has entries $\sigma_{uv}A_{uv}$. We write
$\rhomax(A_\sigma)=\max_i|\lambda_i(A_\sigma)|$. The cycle space
$Z(G)\le\F_2^E$ has dimension $m-n+1$; the \emph{sign} of $z\in Z(G)$
under $\sigma$ is $(-1)^{s\cdot z}$, and a cycle $C$ is \emph{unbalanced}
if $s\cdot\chi_C=1$. Switching preserves the spectrum and acts trivially
on cycle signs; conversely the switching class of $\sigma$ is exactly the
linear functional $z\mapsto s\cdot z$ on $Z(G)$, so there are
$2^{m-n+1}$ switching classes.

For a family $\mathcal C$ of even cycles, call the constraint system
$s\cdot\chi_C=1$ for all $C\in\mathcal C$ \emph{consistent} if it has
at least one solution $s\in\F_2^E$; its solution set is then an affine
parity family, and we write $W=\operatorname{span}_{\F_2}\{\chi_C:C\in\mathcal C\}$.
A graph is \emph{$R$-bicycle-free} if every ball of radius $R$ has cycle
rank at most $1$, where the cycle rank of a graph $H$ is
$|E(H)|-|V(H)|+c(H)$ and $c(H)$ is its number of connected components
(equivalently, every such connected ball is a tree or unicyclic).
Throughout, graph distance and balls are taken in the underlying
unsigned graph.

\section{Even-trace parity}

For a closed walk $W$ of length $\ell$ let $m_e(W)$ be the number of
traversals of $e$, and let $z(W)\in\F_2^E$ be the parity vector
$z(W)_e=m_e(W)\bmod 2$. Since $W$ is closed, every vertex has even total
incidence with the edge multiset of $W$, so the support of $z(W)$ has all
even degrees: $z(W)\in Z(G)$. Moreover
$|z(W)|\equiv\sum_e m_e(W)=\ell\pmod 2$, and the sign of $W$ under
$\sigma$ is $\prod_e\sigma_e^{m_e(W)}=(-1)^{s\cdot z(W)}$.

Let $\Ze=\{z\in Z(G):|z|\text{ even}\}$, the kernel of the linear map
$z\mapsto|z|\bmod2$ on $Z(G)$. If $G$ is bipartite then $\Ze=Z(G)$;
otherwise $\Ze$ has index $2$, hence dimension $m-n$.

\begin{lemma}[Even-trace parity]\label{lem:parity}
For every signing $\sigma$ and every even $\ell$,
\[
\tr\!\big(A_\sigma^{\ell}\big)
=\sum_{z\in \Ze} N_\ell(z)\,(-1)^{s\cdot z},
\]
where $N_\ell(z)\ge0$ is the number of closed $\ell$-walks in $G$ with
parity vector $z$, independent of $\sigma$.
\end{lemma}

\begin{proof}
Group the closed $\ell$-walks by $z(W)$ and use
$\mathrm{sign}(W)=(-1)^{s\cdot z(W)}$. For even $\ell$ only classes with
$|z|$ even occur, by the congruence above.
\end{proof}

\begin{corollary}\label{cor:evenchar}
$\rhomax(A_\sigma)$ is a function of the restriction of the character
$z\mapsto(-1)^{s\cdot z}$ to $\Ze$. In particular
$\rhomax(A_{-\sigma})=\rhomax(A_\sigma)$, and for non-bipartite $G$ the
map from switching classes to $\rhomax$-relevant data is $2$-to-$1$: at
most $2^{m-n}$ values of $\rhomax$ occur among the $2^{m-n+1}$ classes.
\end{corollary}

\begin{proof}
$\rhomax(A_\sigma)=\lim_{k\to\infty}\tr(A_\sigma^{2k})^{1/2k}$ since
$\rhomax^{2k}\le\tr(A_\sigma^{2k})\le n\rhomax^{2k}$, and by
Lemma~\ref{lem:parity} every even trace depends on $s$ only through
$s|_{\Ze}$. Negating $\sigma$ adds $|z|\bmod 2$ to $s\cdot z$, which
vanishes on $\Ze$.
\end{proof}

\begin{corollary}[Interference accounting]\label{cor:interference}
Let $C$ be a cycle of $G$.
If $|C|$ is even, the parity class $z=\chi_C$ contributes to even traces
with the sign $(-1)^{s\cdot\chi_C}$ of $C$ itself. If $|C|$ is odd, then
$\chi_C\notin\Ze$ and $C$ cannot occur alone in an even trace; the basic
combined classes are $z=\chi_{C_1}+\chi_{C_2}$ for pairs of odd cycles,
contributing with sign equal to the \emph{product} of the two cycle signs.
\end{corollary}

For an individual even cycle, prescribing a negative cycle sign
gives a negative contribution from its parity class. For a pair of odd
cycles with identical signs, the product of their signs is positive.
The sizes of the corresponding walk counts $N_\ell(z)$ determine how
much either contribution changes the trace.

\section{\texorpdfstring{An $\F_2$ obstruction and the parameter $\beta_L$}{A parity obstruction and the parameter beta L}}

Fix $L$ and consider the linear system over $\F_2$
\begin{equation}\label{eq:sys}
s\cdot\chi_C=1\qquad\text{for every even cycle $C$ of length}\le L .
\end{equation}

\begin{lemma}[$K_4$ obstruction]\label{lem:k4}
If $G$ contains $K_4$ as a subgraph, then the system \eqref{eq:sys} with
$L\ge4$ is inconsistent: no signing of $G$ makes all quadrilaterals
unbalanced.
\end{lemma}

\begin{proof}
Let $\{a,b,c,d\}$ span a $K_4$ and let $C_1,C_2,C_3$ be its three
$4$-cycles. Every edge of the $K_4$ lies in exactly two of them, so
$\chi_{C_1}+\chi_{C_2}+\chi_{C_3}=0$ in $\F_2^E$, whence
$s\cdot\chi_{C_1}+s\cdot\chi_{C_2}+s\cdot\chi_{C_3}=0$ for every $s$,
while \eqref{eq:sys} demands the sum be $1+1+1=1$.
\end{proof}

More generally, any odd-cardinality dependent family of even short
cycles blocks \eqref{eq:sys}. Let $\mathcal C_L(G)$ denote the set of even
cycles of length at most $L$. If $\mathcal C_L(G)\ne\varnothing$, define
\[
\beta_L(G)\;=\frac{1}{|\mathcal C_L(G)|}\;\max_{s\in\F_2^E}
|\{C\in\mathcal C_L(G):s\cdot\chi_C=1\}|,
\]
the maximum simultaneously-unbalanceable fraction; if
$\mathcal C_L(G)=\varnothing$, set $\beta_L(G)=0$. Every explicit
signing certifies a lower bound on $\beta_L$, and $\beta_L\ge\tfrac12$
whenever $\mathcal C_L(G)$ is nonempty, by averaging over uniform
signings. For a chosen consistent subsystem, row reduction produces its affine
solution family; its dimension is determined by that subsystem.
Computing $\beta_L$ exactly is a MAX-XOR-SAT-type optimization problem;
we do not use any complexity-theoretic claim about this restricted
cycle-incidence family below. Our data
(\S\ref{sec:experiments}) suggest $\beta_L$ as a useful candidate parameter for the spectral
defect $\min_\sigma\rhomax(A_\sigma)-2\sqrt{d-1}$ on cycle-dense instances.

\section{The hypercube: an exact certificate}

\begin{proposition}\label{prop:cube}
On $Q_n$ the system \eqref{eq:sys} with $L=4$ is consistent, and
\emph{every} solution $\sigma$ satisfies $A_\sigma^2=nI$; consequently
$\rhomax(A_\sigma)=\sqrt n$ for the entire affine solution family.
\end{proposition}

\begin{proof}
Consistency is witnessed by Huang's signing \cite{huang19}. Let $\sigma$
be any solution. $Q_n$ is triangle-free, so
$(A_\sigma^2)_{uv}=0$ for adjacent $u,v$. If $u\ne v$ are at distance
$2$ they have exactly two common neighbours $w,w'$, and
$u\,w\,v\,w'$ is a $4$-cycle; unbalancedness gives
$\sigma_{uw}\sigma_{wv}=-\sigma_{uw'}\sigma_{w'v}$, so the two terms of
$(A_\sigma^2)_{uv}$ cancel. Pairs at distance $\ge3$ contribute nothing,
and the diagonal equals the degree $n$.
\end{proof}

The family on $Q_n$ therefore has zero variance in spectral radius.
The pairwise cancellation in this proof uses the two common neighbors
of every distance-two pair. An open question is how much cancellation
persists when only a fraction of the quadrilaterals can be made
unbalanced.

\section{\texorpdfstring{Failure of two-sided interlacing for $A_\sigma^2$}{Failure of two-sided interlacing for A sigma squared}}

By Godsil-Gutman \cite{godsilgutman81},
$\mathbb{E}_\sigma\det(xI-A_\sigma)=\mu_G(x)$, the matching polynomial,
which is real-rooted with roots in $[-2\sqrt{d-1},2\sqrt{d-1}]$
\cite{heilmannlieb72}; this underlies \cite{mss15}. A natural two-sided
analogue would take
$q_G(x):=\mathbb{E}_\sigma\det\!\big(xI-A_\sigma^2\big)$ and hope for
real-rootedness plus an interlacing family, which would give some
$\sigma$ with $\rhomax(A_\sigma)^2\le\operatorname{maxroot}(q_G)$.

\begin{remark}\label{rem:qG}
$q_G$ is not real-rooted in general. For $G=C_4$ the two switching
classes have $A_\sigma^2$-spectra $\{4,4,0,0\}$ and $\{2,2,2,2\}$, so
\[
q_{C_4}(x)=\tfrac12\big[x^2(x-4)^2+(x-2)^4\big]
=x^4-8x^3+20x^2-16x+8=(x^2-4x+2)^2+4,
\]
which has no real roots at all. On the finite atlas and standard examples
recorded in the accompanying repository, the same exact averaging procedure
finds non-real-rooted $q_G$ for $C_4$, $K_{3,3}$, $Q_3$, and Petersen,
while $q_{K_4}=x^4-12x^3+42x^2-52x+21$ is real-rooted. This is a finite
computational observation, not a classification; exact integer characteristic
polynomials are used before the real-root count. A broader characterization of
graphs with real-rooted $q_G$ remains open.
\end{remark}

The even-subgraph corrections that cancel in
$\mathbb{E}\det(xI-A_\sigma)$ survive squaring; this is the concrete
obstruction to extending the interlacing method two-sidedly over the
plain signing family. Any successful two-sided interlacing extension
must overcome this obstruction.

\section{Computational evidence}\label{sec:experiments}

The script \texttt{verify\_periodic\_scan.py} enumerates periodic signings and
computes eigenvalues numerically. The script \texttt{verify\_conjecture\_28.py}
checks the $n=32$ example using exact integer arithmetic.

\emph{Finite checks.} Fixing the signs on a spanning tree to $+1$ leaves
$2^{m-n+1}$ switching classes. The file \texttt{atlas.csv} records all
14 connected regular graphs of degree at least two on three through
seven vertices, followed by eight named examples. For each graph,
\texttt{run\_atlas.py} enumerates the classes and evaluates their
spectral radii with a floating-point symmetric eigensolver. Thus
enumeration is exhaustive, while the tabulated radii are numerical.
For comparison, \texttt{verify\_all.py} computes the averaged
characteristic polynomials in Remark~\ref{rem:qG} with exact rational
coefficients on five graphs and counts their real roots symbolically.
All 18 finite checks in that script passed on rerun. For even $n\ge10$
the general constrained circulant spectrum is proved in
Section~\ref{sec:circulants}. At $n=8$, two additional quadrilaterals
reduce the all-quadrilateral family to two switching classes. At $n=6$ the three paths
$0$--$1$--$3$, $0$--$2$--$3$ and $0$--$4$--$3$ form three quadrilaterals
in pairs. Their sign products multiply to $+1$, since each edge appears
twice, so all three cannot be unbalanced.

\emph{Dispersion on lift towers.} The file \texttt{variance\_onset\_v2.csv}
contains one row per graph and sampling seed. On random $2$-lift towers
over $Q_3$ and $K_4$ we compared the standard deviation of $\rhomax$ in
the computed affine system with that for independent uniform edge signs.
For $n=256,512,1024$ the 18 rows (two bases and three seeds per size)
give ratios from $0.789$ to $1.207$. Each numerator uses 199 random
free-coordinate draws, with the separately evaluated deterministic
solution omitted. Each denominator uses 100 uniform signings. With
$s=0,1,2$, the graph seeds are $100+s$ for $Q_3$ and $200+s$ for $K_4$;
the signing seed is $s$. This range describes these 18 instances.

\emph{Exploratory search protocol.} The archived \texttt{campaign.py}
specifies five strategies, a budget of 2,000 spectral evaluations per
strategy and instance, and 20 instances in each of three sets:
planted graphs, graphs generated by cycle-count-conditioned switches,
and circulants. The random-affine control matches the rank of the
retained parity constraints. The switch sampler has no established
mixing bound. Output from the earlier unmatched protocol is archived
separately. The trace comparison below instead measures three signing
distributions on the same fixed graph.

\section{The kernel-averaged trace}\label{sec:kernel}

Fix a consistent system \eqref{eq:sys}, let
$W=\operatorname{span}\{\chi_C\}\subseteq\Ze$ be the span of its
constraint cycles, and let $\mathcal F$ be its family of switching
classes.

\begin{proposition}[Master identity]\label{prop:master}
The map $\pi\colon W\to\F_2$,
$\pi\big(\sum_{C\in S}\chi_C\big)=|S|\bmod2$, is a well-defined linear
form, and for every even $\ell$,
\[
\mathbb E_{\sigma\in\mathcal F}\,\tr\!\big(A_\sigma^{\ell}\big)
\;=\;\sum_{z\in W}(-1)^{\pi(z)}\,N_\ell(z).
\]
\end{proposition}

\begin{proof}
Consistency means every dependency $\sum_{C\in S}\chi_C=0$ has $|S|$
even, so $\pi$ is well defined; linearity follows from
$|S|+|S'|\equiv|S\triangle S'|\pmod2$. Identify switching classes with
$\operatorname{Hom}(Z(G),\F_2)$; then
$\mathcal F=\varphi_0+\operatorname{Ann}(W)$ and $\varphi|_W=\pi$ for
every $\varphi\in\mathcal F$. For $z\in Z(G)$ the average
$\mathbb E_{\psi\in\operatorname{Ann}(W)}(-1)^{\psi(z)}$ equals $1$ if
$z\in(\operatorname{Ann}W)^{\perp}=W$ and $0$ otherwise. Averaging the
expansion of Lemma~\ref{lem:parity} term by term gives the claim; note
$W\subseteq\Ze$ since the constraint cycles are even.
\end{proof}

Only the total edge-parity vector of a walk matters. Vectors outside $W$
average to zero. A vector in $W$ represented as a sum of $j$ constraint
vectors contributes with sign $(-1)^j$.

\begin{corollary}\label{cor:certificate}
For every integer $k\ge1$,
\[
\min_{\sigma\in\mathcal F}\rhomax(A_\sigma)\;\le\;
R_{\mathcal F}(k):=\Big(\sum_{z\in W}(-1)^{\pi(z)}N_{2k}(z)\Big)^{1/2k}.
\]

For a $d$-regular graph with $d\ge2$, averaging over all signings retains
only $z=0$. Let $t_{2k}(d)$ count closed walks of length $2k$ from
one root of the infinite $d$-regular tree. Each tree walk uses each
edge an even number of times. Its projection to $G$ determines it
uniquely once the lift of the starting vertex is fixed. Hence
$N_{2k}(0)=\mathbb E_\sigma\tr(A_\sigma^{2k})\ge n\,t_{2k}(d)$.
The affine family improves on this finite tree-moment benchmark,
$R_{\mathcal F}(k)<(n\,t_{2k}(d))^{1/(2k)}$, precisely when
\[
\sum_{z\in W\setminus\{0\}}(-1)^{\pi(z)}N_{2k}(z)
  < -\big(N_{2k}(0)-n\,t_{2k}(d)\big).
\]
A separate comparison is needed to certify the spectral threshold
$2\sqrt{d-1}$. For example, on $K_4$ the uniform fourth trace moment is
$4\cdot3(2\cdot3-1)=60<(2\sqrt2)^4$: a rooted tree walk of length
four has $d^2+d(d-1)$ choices, and a four-step totally even walk in a
simple graph uses at most two edges.

\end{corollary}

\begin{corollary}[Confined mass]\label{cor:balanced}
The homogeneous family $\operatorname{Ann}(W)$, all constraint cycles
\emph{balanced}, measures the total parity-confined walk mass:
\[
\mathbb E_{\psi\in\operatorname{Ann}(W)}\,\tr\!\big(A_\psi^{2k}\big)
=\sum_{z\in W}N_{2k}(z).
\]
The same annihilator computation with $\varphi_0=0$ gives each $z\in W$
weight $+1$. Sampling adjacency spectra in this balanced family estimates
the total confined walk mass. Question~\ref{q:confined} concerns the
corresponding primitive non-backtracking classes.
\end{corollary}

The trace in Corollary~\ref{cor:certificate} sums over all $n$
eigenvalues. Dividing the trace by $n$ would divide its $2k$th root by
$n^{1/(2k)}$. To compare with the finite tree benchmark at $k=30$,
we measured
$Y=\tr(A_\sigma^{60})/[n\,t_{60}(4)]$
on two fixed $4$-regular graphs. For $n=256,512$,
\texttt{trace\_rates.py} starts from a random regular graph (generator
seed 300) and requests $n/2$ successful quadrilateral-planting swaps.
The resulting graphs contain 69 and 143 quadrilaterals, respectively.
When the full constraint system is inconsistent, row reduction retains
its independent pivot equations; uniform free coordinates sample this
affine subsystem. The sampling seed is 5.

\begin{center}
\small\setlength{\tabcolsep}{4pt}\begin{tabular}{r r r r r}
$n$ & Draws per distribution & Mean $Y$ & MC standard error
& $\widehat R_{\rm bal}-\widehat R_{\rm all}$\\
\hline
256 & 500 & 0.08041 & 0.00073 & 0.1200\\
512 & 300 & 0.10677 & 0.00065 & 0.1262
\end{tabular}
\end{center}
The columns headed mean and MC standard error give the affine-sample
mean of $Y$ and its sample standard deviation divided by the square
root of the number of draws. The last column subtracts the estimated uniform-signing
trace root from that for the homogeneous balanced system, using the
same draw count for each distribution. All three estimates use
floating-point eigenvalues. Sampling uncertainty here is conditional
on a single graph at each size. A sampled mean does not furnish the
exact family average in the corollary.

\section{\texorpdfstring{A kernel-averaged $L$-function}{A kernel-averaged L-function}}\label{sec:zeta}

In this section $G$ is $d$-regular with $d\ge2$. On the directed
edges of $G$, let $(B_\sigma)_{ef}=\sigma_f$ when $f$ legally follows
$e$ without immediately reversing it, and zero otherwise.
The unsigned operator $B$ corresponds to the all-positive signing.
The signed Ihara product~\cite{bass92,lihouwang24} gives
$\det(I-uB_\sigma)=\prod_{p}\big(1-\sigma(p)u^{|p|}\big)$, the product
running over primes (primitive cyclic classes of closed non-backtracking
walks), where $\sigma(p)=(-1)^{\varphi(z(p))}$ and $z(p)\in Z(G)$ is the
edge-multiplicity parity vector of $p$. Averaging over the family
diagonalizes this product.

\begin{proposition}[Kernel-averaged $L$-identity]\label{prop:Lident}
For primes with $z(p)\in W$ set $\varepsilon_p=(-1)^{\pi(z(p))}$.
With $W$, $\pi$ and $\mathcal F$ as above, define
$\log\det(I-uB_\sigma)$ to be the analytic branch that vanishes at $u=0$.
Then, for $|u|<\tfrac1{d-1}$,
\[
\mathbb E_{\sigma\in\mathcal F}\,\log\det(I-uB_\sigma)
=\tfrac12\log\det(I-u^2B)
-\tfrac12\!\!\sum_{p\,:\,z(p)\in W}\!\!
\log\frac{1+\varepsilon_pu^{|p|}}{1-\varepsilon_pu^{|p|}}\,.
\]
\end{proposition}

\begin{proof}
For $|u|<1/(d-1)$ the Euler product and its logarithm converge
absolutely. With the branch fixed by $\log\det(I-uB_\sigma)=0$ at $u=0$,
the character $(-1)^{\varphi(z)}$ is fixed at $(-1)^{\pi(z)}$ on $\mathcal F$ when
$z\in W$, whereas for $z\notin W$ its average over the affine family is zero. Hence
$\mathbb E\log(1-\sigma(p)u^{|p|})$ equals
$\log(1-\varepsilon_pu^{|p|})$ for parity-confined primes and
$\tfrac12\log(1-u^{2|p|})$ for the rest. This uses only linearity of expectation;
no independence between the signs of different primes is assumed. Collecting the
factor over all primes gives
\emph{all} primes gives
$\tfrac12\log\prod_p(1-(u^2)^{|p|})=\tfrac12\log\det(I-u^2B)$, and the
per-prime correction for confined primes is
$\log(1-\varepsilon u^{\ell})-\tfrac12\log(1-u^{2\ell})
=\tfrac12\log\frac{1-\varepsilon u^{\ell}}{1+\varepsilon u^{\ell}}$.
Absolute convergence for $|u|(d-1)<1$ follows from
$\#\{p:|p|=\ell\}\le\tfrac{2m}{\ell}(d-1)^{\ell}$.
\end{proof}

Both sides transfer to the adjacency side member-wise via
$\det(I-uB_\sigma)=(1-u^2)^{m-n}\det(I-uA_\sigma+u^2(d-1)I)$. On the triangular prism the three quadrilateral constraints leave two
switching classes. The script \texttt{prism\_L.py} enumerates 2,514
oriented primitive cyclic classes of length at most 14; 1,598 have
parity in $W$. For $1\le k\le14$, the two-class average of
$\tr B_\sigma^k$ agrees with the count from the primitive cycles within
the script's absolute tolerance $10^{-6}$. At $u=0.2$ the direct
log determinant and the expression with the prime correction truncated
at length 14 both round to $0.00909552$. The finite coefficient check
and the truncated numerical comparison complement the analytic proof.

The structural content: the untwisted term $\tfrac12\log\det(I-u^2B)$
extends analytically to the Ramanujan disk $|u|<(d-1)^{-1/2}$; primes
whose parity escapes $W$ are averaged into this untwisted contribution,
while the remaining correction is carried by parity-confined primes. At the level of the averaged $L$-function, the sign problem has
become a counting problem. We emphasize the scope of the identity: the
domain of analyticity of $\mathbb E\log\det$ is governed by the worst
member of $\mathcal F$, so it does not by itself bound
$\min_{\sigma\in\mathcal F}\rho$; the quantitative route of this
paper is Corollary~\ref{cor:certificate}, and the identity's role is
to explain the mechanism and to isolate Question~\ref{q:confined}.

\begin{question}[Confined-prime counting]\label{q:confined}
Fix $d\ge3$ and $\delta>0$. For which sequences of $d$-regular graphs
$G_j$, constraint spans $W_j$, and length ranges $I_j$ does a constant
$C$ independent of $j$ and $\ell$ give
\[
 \#\{p: |p|=\ell,\ z(p)\in W_j\}
 \le C(\sqrt{d-1}+\delta)^\ell,\qquad \ell\in I_j?
\]
Here $p$ is a primitive cyclic class of oriented non-backtracking closed
walks. A totally even prime has $z(p)=0$. On each fixed graph this
slice has exponential rate $d-1$ along even lengths
(Remark~\ref{rem:fixed-even}). The length range therefore matters
when seeking a smaller uniform rate. Bounds for the nonzero confined
classes also bear on the cancellation in Proposition~\ref{prop:master}.

\end{question}

\section{\texorpdfstring{The $z=0$ slice: totally even walks}{The z=0 slice: totally even walks}}\label{sec:even}

We count rooted, oriented, cyclically non-backtracking closed walks
whose edge multiplicities are all even. This count includes proper
powers. In particular, traversing a closed walk twice gives a lower
bound for the walk count. To pass to the primitive cyclic classes in
Question~\ref{q:confined}, we remove proper powers and account for
cyclic changes of the starting edge.

\begin{proposition}[Slice identity and doubling]\label{prop:slice}
Let $\mathrm{Ev}_\ell$ denote the number of totally even cyclically
non-backtracking closed walks of length $\ell$ in $G$. Then
$\mathrm{Ev}_\ell=\mathbb E_{\sigma}\,\tr(B_\sigma^{\ell})$, the average
over all signings. For every integer $q\ge1$, the doubling map
$w\mapsto w^2$ injects cyclically non-backtracking closed walks of length
$q$ into totally even walks of length $2q$, so
\[
\mathrm{Ev}_{2q}\ge\tr(B^q).
\]
If $G$ is a fixed finite connected $d$-regular graph with $d\ge3$, then
\[
\liminf_{q\to\infty}\mathrm{Ev}_{4q}^{1/(4q)}\ge\sqrt{d-1}.
\]
\end{proposition}

\begin{proof}
$\tr(B_\sigma^\ell)$ is the signed count of cyclically non-backtracking
closed walks, the sign being $\prod_e\sigma_e^{m_e(w)}$; averaging kills every
walk with an odd multiplicity. If $w$ is cyclically non-backtracking
then so is $w^2$ (the seam repeats a legal transition of $w$), its
multiplicities double, and its first half recovers $w$. For the growth
claim put $b=d-1>1$. The unsigned Ihara--Bass determinant formula
gives the eigenvalues of $B$ from the roots of $y^2-\lambda y+b$
for $\lambda\in\operatorname{spec}(A)$ and additional roots $\pm1$.
Nonreal roots have modulus $\sqrt b$. Among the real roots only
$y=\pm b$ can have modulus $b$, arising from $\lambda=\pm d$,
and every root has modulus at most $b$. Connectedness makes $d$ simple,
and $-d$ occurs precisely in the bipartite case, also simply. Thus
$\tr(B^{2q})=(c+o(1))b^{2q}$ with $c=2$ for bipartite $G$ and $c=1$
otherwise. Apply the doubling inequality with $2q$ in place of $q$.
For degree two, cycle lengths impose additional divisibility conditions:
on $C_3$, $\mathrm{Ev}_{4q}=0$ whenever $3\nmid q$.
\end{proof}

\begin{remark}[Fixed-graph growth of the totally even slice]\label{rem:fixed-even}
Assume $G$ is a fixed connected $d$-regular graph with $d\ge3$, and let
$r=|E|-|V|+1$ be its cycle rank. For every even $\ell$,
\[
0\le \mathrm{Ev}_\ell\le \operatorname{tr}(B^\ell).
\]

Partition the uniform signing space into $2^r$ equally likely switching
classes. Put $b=d-1$. The signed Ihara--Bass formula gives quadratic
factors $y^2-\lambda y+b$ with $\lambda\in\operatorname{spec}(A_\sigma)$
and additional roots $\pm1$. As in Proposition~\ref{prop:slice}, a root
of modulus $b$ requires $\lambda=\pm d$. If $A_\sigma v=\pm d v$, choose
$v$ real and a vertex maximizing $|v_i|$. Equality in the triangle
inequality there forces every neighbor to have the same absolute value.
Connectedness propagates equality throughout $G$, with
$\sigma_{ij}=\pm\operatorname{sgn}(v_i)\operatorname{sgn}(v_j)$ for
every edge and the same overall sign throughout. Thus precisely the
switching classes of the all-plus and all-minus signings contribute
peripheral eigenvalues. For non-bipartite $G$ these classes are distinct
and each contributes one simple eigenvalue of modulus $b$. For bipartite
$G$ the two classes coincide and contribute the simple pair $b,-b$.
The remaining eigenvalues over the finitely many classes have moduli
at most some $\rho_2<b$. Consequently, for even $\ell$,
\[
\mathrm{Ev}_\ell=2^{1-r}b^\ell+O(|E|\rho_2^\ell).
\]
To obtain the primitive count, observe that a rooted proper power of
length $\ell$ is determined by a rooted closed walk of length
$j\le\ell/2$ dividing $\ell$. There are at most $2|E|b^{j-1}$
rooted non-backtracking walks of length $j$ that close, so the number
of rooted proper powers is $O_G(b^{\ell/2})$, summing the geometric
series. Every oriented primitive cyclic class contributes exactly
$\ell$ rooted walks. The number of totally even primitive classes
of even length $\ell$ is therefore asymptotic to
$2^{1-r}b^\ell/\ell$ as $\ell\to\infty$ on this fixed graph.

\end{remark}

\begin{proposition}[Dense trapping]\label{prop:trap}
Let $d\ge4$ and suppose $G$ contains $K_d$ as a subgraph (realizable in
connected $d$-regular graphs: two copies of $K_d$ joined by a perfect
matching). Then there are $c_d>0$ and $\ell_0(d)$, independent of $G$
and $n$, with $\mathrm{Ev}_\ell(G)\ge c_d\,(d-2)^{\ell}$ for all even
$\ell\ge\ell_0$. Since $d-2>\sqrt{d-1}$ for every $d\ge4$, the confined
mass exceeds the profile $n(\sqrt{d-1}\,)^{\ell}$ for every even
$\ell\ge C_d\log n$.
\end{proposition}

\begin{proof}
Here $d$ denotes the degree of the ambient regular graph $G$; the clique $K_d$ itself is
$(d-1)$-regular, so its unsigned non-backtracking operator has Perron value $d-2$.
Totally even walks of the $K_d$ copy are totally even walks of $G$, and
by Proposition~\ref{prop:slice} their number is
$2^{-|E(K_d)|}\sum_{\sigma\in\{\pm1\}^{E(K_d)}}\tr(B_{K_d,\sigma}^\ell)$.
The unsigned $B_{K_d}$ is irreducible and aperiodic with Perron value
$d-2$; its remaining eigenvalues have modulus at most
$\max(1,\sqrt{d-2})<d-2$ (Ihara-Bass from
$\operatorname{spec}A(K_d)=\{d-1,(-1)^{d-1}\}$). The unsigned spectral radius bounds that of each signed
matrix by Perron--Frobenius comparison~\cite[Chap.~8]{hornjohnson13}.
Applying the signed quadratic-factor argument in
Remark~\ref{rem:fixed-even} to the $(d-1)$-regular graph $K_d$ shows that
equality occurs exactly for the all-plus and all-minus switching classes. Hence $\rho_2:=\max$ over the remaining classes is
strictly below $d-2$, the two trivial classes contribute
$\tr B^\ell+(-1)^\ell\tr B^\ell=2(d-2)^\ell(1+o(1))$ at even $\ell$, and
the class average is at least
$2^{-r}\big[2(d-2)^\ell(1-o(1))-2^{r}\,d(d-1)\rho_2^{\ell}\big]$ with
$r=\binom d2-d+1$.
\end{proof}

\begin{remark}[Necessity of an unsigned core restriction]
Since $0\in W$ always, Proposition~\ref{prop:trap} shows that an unconditional upper
bound on the totally even mass at the Ramanujan rate cannot hold once a fixed dense core
is allowed to contribute for lengths of order $\log n$. In particular, the obstruction is
carried by the all-plus and all-minus switching classes of the core, whose non-backtracking
spectral radius is the unsigned value $d_H-1$. This is why a successful counting hypothesis
must constrain unsigned local growth, rather than merely exclude a subset of nontrivial
signing classes. The finite atlas also illustrates the distinction between the two operators:
for $K_4$ and Petersen, nontrivial classes can have adjacency spectral radius strictly below
the Kesten adjacency benchmark. The corresponding non-backtracking statement is a separate
spectral fact and should not be conflated with the adjacency comparison.
\end{remark}

The following finite-scale counting condition is useful for the
bounded-rank part of the method. Say $G$ is \emph{$(r_0,\delta)$-subcritical
at scale $\ell$} if every connected subgraph $H$ of $G$ with minimum degree
$\ge2$ and at most $\ell$ edges satisfies: (i) the cycle rank of $H$
is at most $r_0$; and (ii) the number $c_H(t)$ of closed non-backtracking walks of length $t$
(the initial/final step need not satisfy the cyclic condition) obeys
$c_H(t)\le 2|E(H)|\,(t+1)^{3r_0}(1+\delta)^{t}$ for all $t\le8\ell$.
The notion is monotone: subcriticality at scale $\ell$ implies it
at every smaller scale, and all subsequent hypotheses are stated
through this single definition.
Condition (ii) holds, for instance, whenever every thread of $H$ (a
maximal path through degree-$2$ vertices) has length at least
$\lambda$, with $1+\delta=(2r_0-1)^{1/\lambda}$: closed walks in $H$
project to reduced cyclic words in a free group of rank $\le r_0$ whose
generators cost at least $\lambda$ steps each.

\begin{lemma}[Non-backtracking reachability]\label{lem:reach}
Let $H$ be connected with minimum degree $\ge2$, $s$ edges, and not a
cycle. Between any two directed edges of $H$ there is a
non-backtracking path of length at most $2s-1$.
\end{lemma}

\begin{proof}
Let $\overrightarrow H$ be the oriented line graph whose vertices are the
$2s$ directed edges of $H$, with an arc $e\to f$ when $t(e)=o(f)$ and
$f\ne\bar e$. For a connected graph of minimum degree at least $2$ that
is not a cycle, $\overrightarrow H$ is strongly connected \cite[Sec.~3, p.~14]{ks00}.
Hence between any two of its vertices there is a simple directed path of
length at most $2s-1$, which is exactly the required non-backtracking path in $H$.
\end{proof}

\begin{lemma}[Fresh-run rank increase]\label{lem:fresh}
Let $w$ be a cyclically non-backtracking closed walk, and expose a chosen
initial directed edge before declaring edges fresh. Let $S$ be the support
of $w$, and let $q$ be the number of maximal nonempty fresh runs. Then
\[
q\le \operatorname{rank}(S).
\]
\end{lemma}

\begin{proof}
Expose the initial directed edge first, so the current support $T$ is nonempty and connected.
For each maximal fresh run, let $F$ be the set of edges exposed during that run and let
$T'$ be the support immediately after the run. We claim
\[
\operatorname{rank}(T')-\operatorname{rank}(T)\ge1.\tag{*}
\]
The run starts at a vertex of $T$, and every edge in $F$ is new when it is exposed.

Suppose first that the fresh trail meets $T$ again before it terminates. Take its first
return to $T$. The subtrail from its starting vertex to that first return has no internal
vertex in $T$ and consists entirely of fresh edges. Adding this nonempty path to the
connected graph $T$ creates a cycle, so the cycle rank increases by at least one. The same
conclusion holds when the fresh run itself terminates in $T$.

Now suppose that the fresh run terminates at a vertex outside $T$. It cannot be the final
run of the closed walk, because the final vertex is the initial vertex, which belongs to $T$.
Hence the next step is stale. Its initial vertex lies outside $T$, while every edge of $T$ has
both endpoints in $T$, so this stale edge cannot belong to $T$. It must therefore be an edge
already exposed earlier in the same fresh run. Choose the first such reuse. Before that reuse
all edges exposed by the run are distinct. Because the walk is non-backtracking, the reused
edge is not the immediately preceding edge; the segment between its two occurrences therefore
contains a nontrivial closed trail using only edges exposed by this run. Its support contains
a cycle, and again the cycle rank increases by at least one, proving $(*)$.

Different maximal fresh runs use disjoint sets of fresh edges. Applying $(*)$ successively
and telescoping the rank increments gives
\[
\operatorname{rank}(S)-\operatorname{rank}(T_0)\ge q,
\]
where $T_0$ is the one-edge initial support and has rank $0$. Hence
$q\le\operatorname{rank}(S)$.
\end{proof}

\begin{lemma}[Closing a stale segment]\label{lem:stale}
Let $S$ be connected with minimum degree at least $2$, $s$ edges, and not
a cycle. Let $P=(e_1,\ldots,e_t)$ be a non-backtracking path in $S$.
Then, for some $a\in\{0,\ldots,2s-1\}$, $P$ is the initial segment of a
closed non-backtracking walk of length $t+a$. Consequently, for fixed $t$,
the number of such paths is at most
\[
2s\sum_{a=0}^{2s-1}c_S(t+a),
\]
where $c_S(r)$ denotes the number of closed non-backtracking walks of
length $r$ in $S$.
\end{lemma}

\begin{proof}
Write $e_1=(v_0,v_1)$ and $e_t=(v_{t-1},v_t)$. Since $S$ has minimum degree at least
$2$, choose a directed edge $h$ entering $v_0$ with $h\ne\bar e_1$. By
Lemma~\ref{lem:reach}, there is a directed path in the oriented line graph
\[
e_t=e^{(0)},e^{(1)},\ldots,e^{(a)}=h,
\qquad 0\le a\le2s-1.
\]
Every transition along this path is non-backtracking. Appending
$e^{(1)},\ldots,e^{(a)}$ to $P$ therefore preserves non-backtracking, and the final edge
$h\ne\bar e_1$ makes the cyclic seam back to $e_1$ legal. We obtain a closed
non-backtracking walk of length $t+a$ whose first $t$ edges are exactly $P$.

To count such paths, choose once and for all a shortest admissible line-graph completion for
each terminal directed edge. For fixed $t$, $a$, and initial directed edge $e_1$, the first
$t$ edges of the completed closed walk recover $P$ uniquely. Thus the completion map is
injective. There are at most $2s$ choices of the initial directed edge and
$0\le a\le2s-1$, giving
\[
\#\{P\}\le2s\sum_{a=0}^{2s-1}c_S(t+a).
\]
The outside factor $2s$ is deliberately crude: $c_S(t+a)$ already counts closed walks with
all possible starting directed edges. Retaining this harmless overcount gives a uniform bound
that is convenient in Proposition~\ref{prop:upper}.
\end{proof}

\begin{proposition}[Even-walk upper bound, bounded-rank regime]\label{prop:upper}
If $G$ is $d$-regular and $(r_0,\delta)$-subcritical at scale $\ell$,
then
\[
\mathrm{Ev}_\ell(G)\;\le\;
nd\,(C\ell)^{\,4(r_0+1)^2}\,
\Big((1+\delta)^{2(r_0+2)}\sqrt{d-1}\Big)^{\ell}
\]
for an absolute constant $C$. The statement is intended only at the stated
finite scale; no asymptotic conclusion in $\ell$ is drawn from it.
\end{proposition}

\begin{proof}
Fix the starting directed edge ($\le nd$ choices) and expose the walk
step by step; let $S_i$ be the subgraph of edges used through step $i$
and $S=S_\ell$ the support. Since multiplicities are even,
$s:=|E(S)|\le\ell/2$; a degree-$1$ vertex of $S$ would force a
backtrack, so $S$ has minimum degree $2$. By hypothesis (i),
$r:=\operatorname{rank}(S)\le r_0$.

By Lemma~\ref{lem:fresh}, the number $q$ of maximal fresh runs satisfies
$q\le r\le r_0$. These runs are separated by at most $q+1\le r_0+1$ nonempty stale
segments. For fixed $q$, specifying the beginning and end of each fresh run uses at most
$2q$ time indices. Padding this list with zeros to length $2r_0$ bounds the total number
of schedules by $(\ell+1)^{2r_0}$.
Once the incoming directed edge is fixed, each fresh edge has at most $d-1$ choices, giving
$(d-1)^s$.

To count branches, fix the final edge count $s$, a run schedule and a
prefix that admits at least one completion with these data. Let $T$ be
the support of the prefix immediately before a stale segment of length
$t$. Every admissible choice for the segment uses only edges in $T$:
the set of exposed edges stays fixed throughout the segment. Select one
completion of this prefix and call its support $S'$. All candidate
segments lie in $T\subseteq S'$, including those that will ultimately
lead to different completions. The graph $S'$ has $s$ edges and minimum
degree at least two, so subcriticality applies. If $S'$ is not a cycle,
Lemma~\ref{lem:stale} bounds the number of candidate segments by

\[
2s\sum_{a=0}^{2s-1}c_{S'}(t+a)
\le C s^2(2\ell+1)^{3r_0+1}(1+\delta)^{t+2s-1},
\]
because $t+a\le2\ell-1$. If $S'$ is a single cycle, then after the initial directed edge is
fixed the non-backtracking continuation is forced, and the contribution is at most $2s$, which
is absorbed by the same polynomial factor. There are at most $r_0+1$ stale segments, and their
total length is exactly the number of non-fresh steps,
\[
(\ell-1)-(s-1)=\ell-s.
\]
The bound holds for every admissible prefix. Multiplying the successive branch bounds gives
\[
\mathrm{Ev}_\ell\le nd\,(\ell+1)^{2r_0}
\sum_{s\le\ell/2}
\big(C s^2(2\ell+1)^{3r_0+1}\big)^{r_0+1}
(d-1)^s(1+\delta)^{(\ell-s)+(2s-1)(r_0+1)}.
\]
The exponent of $1+\delta$ is
\[
E(s)=(\ell-s)+(2s-1)(r_0+1)
=\ell+s(2r_0+1)-(r_0+1).
\]
It is increasing in $s$, and for $s\le\ell/2$,
\[
E(s)\le\ell+\frac{\ell}{2}(2r_0+1)<2(r_0+2)\ell.
\]
Also $(d-1)^s\le(d-1)^{\ell/2}$. Finally, using $s\le\ell/2$ and the at most $\ell/2$
possible values of $s$, the polynomial exponent is bounded by
\[
2r_0+2(r_0+1)+(3r_0+1)(r_0+1)+1
=(3r_0+2)(r_0+2)\le4(r_0+1)^2.
\]
After enlarging the absolute constant $C$, all remaining factors are absorbed into
$(C\ell)^{4(r_0+1)^2}$. This proves the displayed bound.
\end{proof}

\begin{remark}[Scope of the bounded-rank regime]\label{rem:scope}
The subcritical definition and Proposition~\ref{prop:upper} are used only as finite-scale
bounded-rank counting statements. They are not intended as an asymptotic fixed-$r_0$
regime at the scale $\ell\asymp(\log n)^2$. Indeed, for fixed $d\ge3$ and fixed $r$,
an elementary BFS-surplus argument finds, in every sufficiently large connected
$d$-regular graph, a connected minimum-degree-$2$ subgraph of cycle rank $r$ using
$O_{d,r}(\log n)$ edges. Thus the bounded-rank condition with fixed $r_0$ eventually
fails at every sufficiently large multiple of $\log n$, and in particular cannot support
an asymptotic theorem at the $\log^2 n$ scale. We record Proposition~\ref{prop:upper}
for the finite-scale encoding mechanism only.
\end{remark}

\begin{proposition}[Conditioning on an affine parity family]\label{prop:bf-conditioning}
Fix $d\ge3$ and $\varepsilon>0$. Let $(G_n)$ be a sequence of $d$-regular graphs on
$n$ vertices such that there are radii $R_n$ with
$R_n/(\log\log n)^2\to\infty$ and every radius-$R_n$ ball has cycle rank at most one.
Let $\mathcal F_n$ be any consistent affine parity family whose constraint span has
$\dim W_n\le K$ for a fixed constant $K$. Then, for all sufficiently large $n$, the family
contains a signing $\sigma$ with
\[
\rho(A_\sigma)\le2\sqrt{d-1}+\varepsilon.
\]
If, in addition, every prescribed cycle has length $<R_n$ and
$\Lambda_{W_n}:=\max_{z\in W_n}|z|\le\Lambda$ uniformly, then the prescribed
constraint cycles are pairwise edge-disjoint and
$\dim W_n\le\Lambda/4$.
This is a conditioning corollary of the random-signing theorem, not a new
concentration theorem.
\end{proposition}

\begin{proof}
Let $\widetilde{\mathcal F}_n\subseteq\{\pm1\}^{E(G_n)}$ be the full preimage of the
switching-class family $\mathcal F_n$. The affine cycle-sign constraints have rank
$k_n=\dim W_n$, so $\widetilde{\mathcal F}_n$ is a union of switching cosets and has
relative density
\[
p_n:=\Pr_{\sigma\,\mathrm{uniform}}(\sigma\in\widetilde{\mathcal F}_n)=2^{-k_n}\ge2^{-K}.
\]
Equivalently, the family has codimension $k_n$ in the cycle-sign quotient; the same
relative density is obtained in the full edge-sign space because every switching class has
the same number of edge-sign representatives. By the theorem of Mohanty, O'Donnell, and
Paredes~\cite[Thm.~1.2]{mop20}, under the stated bicycle-free hypothesis the probability of the event
$\rho(A_\sigma)>2\sqrt{d-1}+\varepsilon$ tends to zero. Hence, for all sufficiently
large $n$, the bad event has probability strictly smaller than $p_n$ and cannot
contain the whole affine family. Thus some signing in $\mathcal F_n$ satisfies the
required bound.

For the final assertion, suppose two distinct prescribed cycles shared a vertex. Since each
has length $<R_n$, both would lie in the same radius-$R_n$ ball, which would then contain two
independent cycle vectors, contradicting the bicycle-free hypothesis. Thus the prescribed
cycles are pairwise edge-disjoint. Their incidence vectors are therefore linearly independent,
so their number is $k_n=\dim W_n$. Now take the sum of all $k_n$ prescribed cycle vectors:
\[
z_*:=\sum_{i=1}^{k_n}\chi_{C_i}\in W_n.
\]
Because the cycles are edge-disjoint and every cycle has length at least $4$,
\[
|z_*|=\sum_{i=1}^{k_n}|C_i|\ge4k_n.
\]
Since $|z_*|\le\Lambda_{W_n}\le\Lambda$, we obtain
$\dim W_n=k_n\le\Lambda/4$.

\end{proof}

\begin{remark}[Scope relative to the reported cubic counterexample]
The finite cubic counterexample reported by Xu~\cite{xu26counter} is outside the
bicycle-free hypothesis here. In the construction described there, each completed
triple of leaves is joined to two new vertices adjacent to all three leaves, producing
a copy of $K_{2,3}$; this subgraph has cycle rank $2$ and diameter $2$. Thus a radius-$2$
ball already contains two independent cycles. The reported counterexample therefore does
not test the conditioning corollary above, whose hypothesis requires bicycle-freeness at
radii tending to infinity.
\end{remark}

\section{Signed circulants: exact parity families and periodic counterexamples to constrained optimality}\label{sec:circulants}

The circulant $C_n(1,2)$ has vertex set $\mathbb Z_n$ and edges
$\{i,i{+}1\}$ (step $1$) and $\{i,i{+}2\}$ (step $2$); it is
$4$-regular, with Kesten bound $2\sqrt3\approx3.464$. Its
quadrilaterals, for $n\ge10$, are exactly
$Q_i=(i,\,i{+}1,\,i{+}3,\,i{+}2)$, one for each $i\in\mathbb Z_n$, and
its triangles are $T_i=(i,\,i{+}1,\,i{+}2)$. At $n=8$ the step-$2$
edges close into two further quadrilaterals, $(0,2,4,6)$ and
$(1,3,5,7)$. Propositions~\ref{prop:circ} and~\ref{prop:twist}
treat $n\ge10$. Remark~\ref{rem:n8} gives the separate calculation at $n=8$.

\subsection{The quadrilateral parity family}

We apply the affine framework to the quadrilateral constraints on
$C_n(1,2)$. The class count and the two constrained radii are from the
original circulant preprint~\cite{companion}. Writing $Q_i$ for the
quadrilaterals listed above, the system is
\begin{equation}\label{eq:circ-system}
s\cdot\chi_{Q_i}=1\qquad(i\in\mathbb Z_n).
\end{equation}

\begin{proposition}\label{prop:circ}
Let $n\ge10$ be even. On $C_n(1,2)$ the system \eqref{eq:circ-system}
is consistent and its solutions form exactly four switching classes. The
class containing the signing $\sigma\equiv+1$ on step-$1$ edges
$\{i,i{+}1\}$ and $\sigma_{\{i,i+2\}}=(-1)^i$ on step-$2$ edges has
spectrum
\[
\Big\{\pm2\sqrt{\cos^2\theta_k+\cos^2 2\theta_k}\ :\
\theta_k=\tfrac{2\pi k}{n},\ 0\le k<\tfrac n2\Big\},
\]
so its spectral radius equals $2\sqrt2$ exactly, against the Kesten floor
$2\sqrt3$. For odd $n$ the system is inconsistent, and at most $n-1$
quadrilaterals can be simultaneously unbalanced.
\end{proposition}

\begin{proof}
The quadrilaterals of $C_n(1,2)$ are the $Q_i$, one per
vertex, with sign $a_ia_{i+2}\,b_ib_{i+1}$, where $a_i,b_i$ denote the
signs of $\{i,i{+}1\}$ and $\{i,i{+}2\}$. The displayed signing gives
$a_ia_{i+2}b_ib_{i+1}=(-1)^i(-1)^{i+1}=-1$ for all $i$, so the system is
consistent. If $\sum_{i\in S}\chi_{Q_i}=0$, then comparing coefficients
of the step-$2$ edge $\{i,i{+}2\}\in Q_i\cap Q_{i-1}$ forces $S$ to be
shift-invariant, so $S=\emptyset$ or $S=\mathbb{Z}_n$; the full sum is
indeed zero (each edge lies in exactly two quadrilaterals). Hence the
constraint map on the $2^{\,m-n+1}=2^{\,n+1}$ switching classes has rank
$n-1$, and the solution set is a coset of dimension $2$: four classes.

For the displayed class, write $C_1,C_2$ for the unsigned step-$1$ and
step-$2$ circulants and $D=\operatorname{diag}((-1)^j)$; then
$A_\sigma=C_1+DC_2$, which is symmetric since $D$ commutes with the even
shift $C_2$. On the orthonormal Fourier basis
$f_k(j)=n^{-1/2}e^{2\pi ijk/n}$ we have $C_1f_k=2\cos\theta_k\,f_k$,
$C_2f_k=2\cos2\theta_k\,f_k$, and $Df_k=f_{k+n/2}$. Each plane
$\operatorname{span}\{f_k,f_{k+n/2}\}$ is therefore invariant, carrying
the block
\[
\begin{pmatrix} 2\cos\theta_k & 2\cos2\theta_k\\
2\cos2\theta_k & -2\cos\theta_k\end{pmatrix},
\]
where we used $\cos\theta_{k+n/2}=-\cos\theta_k$ and
$\cos2\theta_{k+n/2}=\cos2\theta_k$. A traceless symmetric $2\times2$
block $\left(\begin{smallmatrix}a&b\\ b&-a\end{smallmatrix}\right)$ has
eigenvalues $\pm\sqrt{a^2+b^2}$, giving the stated spectrum. Setting
$u=\cos^2\theta\in[0,1]$,
$\cos^2\theta+\cos^22\theta=4u^2-3u+1\le2$ with equality iff $u=1$,
attained at $k=0$; hence $\rhomax=2\sqrt2$. For odd $n$, summing all $n$
constraints gives $0$ on the left (each edge appears twice) and
$n\equiv1\pmod2$ on the right.
\end{proof}

\begin{lemma}[Holonomy determines diagonal-unitary gauge]\label{lem:holonomy-gauge}
Let $M$ and $N$ be Hermitian matrices with the same connected support graph,
all corresponding nonzero entries of modulus one, and identical holonomy on
every cycle. Then $M$ and $N$ are conjugate by a diagonal unitary.
\end{lemma}

\begin{proof}
Choose a spanning tree and set $u_r=1$ at a root. Define $u_v$ recursively so
that $M_{uv}/N_{uv}=u_u\overline{u_v}$ on every tree edge. Equality of
holonomies on the fundamental cycles gives the same identity on every
non-tree edge. Thus $U=\operatorname{diag}(u_v)$ satisfies $M=UNU^{-1}$.
\end{proof}

\begin{proposition}[Twisted classes]\label{prop:twist}
Let $\tau_i$ denote the
sign of the triangle $T_i=(i,i{+}1,i{+}2)$ and $\alpha$ the sign product
around the step-$1$ Hamilton cycle $H$. Then
$\chi_{Q_i}=\chi_{T_i}+\chi_{T_{i+1}}$, so \eqref{eq:circ-system} is equivalent
to the alternation $\tau_{i+1}=-\tau_i$; the triangles together with $H$
form a basis of the cycle space, so the four solution classes are
coordinatized by $(\tau_0,\alpha)\in\{\pm1\}^2$; and rotation by one
vertex is an automorphism carrying $(\tau_0,\alpha)$ to
$(-\tau_0,\alpha)$. Consequently $\rhomax$ depends on $\alpha$ alone: it
equals $2\sqrt2$ for $\alpha=+1$, and for $\alpha=-1$ it equals
\[
\rho_-(n)\;:=\;2\sqrt{\cos^2(\pi/n)+\cos^2(2\pi/n)}\;<\;2\sqrt2 .
\]
In particular $\min_\sigma\rhomax(A_\sigma)\le\rho_-(n)$ for every even
$n\ge10$.
\end{proposition}

Thus $\rho_-(n)$ is the minimum spectral radius within the
quadrilateral-unbalanced family. The question of whether it is the
minimum over all signings is separate.

\begin{proof}
Each step-$2$ edge $b_i$ lies in $T_i$ only, so
$\sum_{i\in S}\chi_{T_i}=0$ forces $S=\emptyset$; together with the fact
that no combination of triangles has empty step-$2$ support except the
empty one, $\{\chi_{T_0},\dots,\chi_{T_{n-1}},\chi_H\}$ is a basis of
the $(n{+}1)$-dimensional cycle space, giving the coordinates; the
rotation statement is immediate since rotation shifts the alternating
pattern by one. For $\alpha=-1$ it therefore suffices to treat the
representative with $\tau_i=(-1)^i$. Let $R$ be the cyclic shift,
$D=\operatorname{diag}((-1)^j)$, $\varphi=\pi/n$, and set
\[
A(\varphi)\;=\;e^{i\varphi}R+e^{-i\varphi}R^{*}
\;+\;e^{2i\varphi}DR^{2}+e^{-2i\varphi}R^{-2}D ,
\]
a Hermitian matrix supported on $C_n(1,2)$. Its holonomy on the directed
triangle $i\to i{+}1\to i{+}2\to i$ is
$e^{i\varphi}\cdot e^{i\varphi}\cdot(-1)^ie^{-2i\varphi}=(-1)^i$, and on
$H$ it is $e^{in\varphi}=-1$.

The triangle cycles together with the Hamilton cycle generate the
integer cycle lattice: every step-$2$ edge occurs in exactly one triangle,
so subtracting suitable triangle cycles removes all step-$2$ coefficients
and leaves an integer multiple of $H$. Thus agreement of the triangle and
$H$ holonomies implies agreement on every integer cycle, as required by Lemma~\ref{lem:holonomy-gauge}. Hence $A(\varphi)=U A_{(+1,-1)}U^{-1}$, so the two matrices are
isospectral. Note that the step-$2$ phase is forced to be $2\varphi$: this is
exactly the co-shift that keeps the triangle holonomies real, and it is the step
omitted by a naive one-parameter Bloch argument. As in Proposition~\ref{prop:circ}, the
planes $\operatorname{span}\{f_k,f_{k+n/2}\}$ are invariant, with blocks
\[
\begin{pmatrix}
2\cos(\theta_k{+}\varphi)&2\cos(2\theta_k{+}2\varphi)\\
2\cos(2\theta_k{+}2\varphi)&-2\cos(\theta_k{+}\varphi)
\end{pmatrix},
\]
hence spectrum $\pm2\sqrt{g(\theta_k+\pi/n)}$ with
$g(\theta)=\cos^2\theta+\cos^22\theta$ and shifted momenta
$\theta_k+\pi/n\in\{(2k{+}1)\pi/n\}$, which avoid $0$ and $\pi$. Writing
$g=4u^2-3u+1$ with $u=\cos^2\theta$, $g$ decreases on $[0,\theta_0]$ and increases on
$[\theta_0,\pi/2]$, where $\cos2\theta_0=-\tfrac14$, with
$g(\pi/2)=1$ and $g$ symmetric under $\theta\mapsto\pi-\theta$;
since $g(\pi/n)\ge g(\pi/8)=\tfrac{4+\sqrt2}{4}>1$ for every even
$n\ge8$, the maximum over the shifted lattice is attained at $\pi/n$,
giving $\rho_-(n)$. (We state the inequality from $n=8$ so that
Remark~\ref{rem:n8} may invoke it.)
\end{proof}

\subsection{Periodic counterexamples}

\begin{theorem}[A periodic counterexample to constrained-family global optimality]\label{thm:counterexample}
Let $n=8m$ with $m\ge1$. On $C_n(1,2)$, assign $+1$ to every
step-$1$ edge and assign to the step-$2$ edges the period-$8$ pattern
\[
(+,+,-,+,-,-,+,-),
\]
repeated $m$ times. Let $r_*$ be the largest real root of
\[
f(x)=x^4-2x^3-6x^2+12x-4.
\]
Then the resulting signed adjacency matrix satisfies
\[
\rhomax(A_{\sigma^{(8)}})=r_*
=2.793604493334841\ldots .
\]
Moreover,
\[
r_*<\rho_-(n)
=2\sqrt{\cos^2(\pi/n)+\cos^2(2\pi/n)}
\]
for every $n=8m\ge32$. Hence
\[
\boxed{\min_\sigma\rhomax(A_\sigma)<\rho_-(n)}
\]
for every $n=32,40,48,56,\ldots$.
\end{theorem}

The first two step-$2$ signs in this pattern agree. With step-$1$ signs
positive, the quadrilateral $Q_0$ is therefore balanced. This gives a
competitor outside the quadrilateral-unbalanced family, whose exact
minimum is $\rho_-(n)$ for even $n\ge10$.

\begin{proof}[Proof of Theorem~\ref{thm:counterexample}]
Because the signing has period $8$, write $n=8m$ and group the vertices
into $m$ cells of eight consecutive vertices. A Bloch transform
therefore decomposes the signed adjacency matrix into $8\times8$
Hermitian blocks $H(z)$ indexed by the $m$th roots of unity $z$:
\[
H(z)=
\begin{pmatrix}
0&1&1&0&0&0&z^{-1}&z^{-1}\\
1&0&1&1&0&0&0&-z^{-1}\\
1&1&0&1&-1&0&0&0\\
0&1&1&0&1&1&0&0\\
0&0&-1&1&0&1&-1&0\\
0&0&0&1&1&0&1&-1\\
z&0&0&0&-1&1&0&1\\
z&-z&0&0&0&-1&1&0
\end{pmatrix}.
\]
Put $s=z+z^{-1}$. Since $|z|=1$, we have $s\in[-2,2]$, and a direct
calculation gives
\begin{equation}\label{eq:bloch-det}
\begin{aligned}
P(x,s)&:=\det(xI-H(z))\\
&=x^8-16x^6+80x^4-128x^2+38\\
&\qquad+s(-2x^4+16x^2-13)+s^2.
\end{aligned}
\end{equation}
At $s=2$, the determinant in \eqref{eq:bloch-det} factors as
\begin{equation}\label{eq:bloch-factor}
P(x,2)=
\bigl(x^4-2x^3-6x^2+12x-4\bigr)
\bigl(x^4+2x^3-6x^2-12x-4\bigr).
\end{equation}
The
\href{https://github.com/Vaibhavs25/bilu-linial-parity/blob/b37002b020edd5d26065a0ee6ee7cff85cf7814c/verify_theorem26.py}{symbolic verification script}
recomputes $\det(xI-H(z))$ and checks equations~\eqref{eq:bloch-det}
and~\eqref{eq:bloch-factor} in exact arithmetic.

Let $f(x)=x^4-2x^3-6x^2+12x-4$. Since
$f(\sqrt6)=-4$ and
\[
f'(x)=4x^3-6x^2-12x+12,
\qquad
f''(x)=12(x^2-x-1)>0\quad(x\ge\sqrt6),
\]
we have $f'(\sqrt6)=12\sqrt6-24>0$, so $f$ is strictly increasing on
$[\sqrt6,\infty)$ and has a unique root there. Direct evaluation gives
$f(2.79360)<0<f(2.79361)$; denote this root by $r_*$. For $x\ge r_*$,
\[
\frac{\partial P}{\partial s}=2s-2x^4+16x^2-13
\le -2x^4+16x^2-9<0,
\]
so $P(x,s)$ is decreasing in $s$ and hence
$P(x,s)\ge P(x,2)$ for all $s\in[-2,2]$. For $x>r_*$ the first
factor in~\eqref{eq:bloch-factor} is positive, while the difference between the second and
first factors is
\[
4x(x^2-6)>0.
\]
Thus both factors in~\eqref{eq:bloch-factor} are positive, and consequently
$P(x,s)>0$ for every $x>r_*$ and every $s\in[-2,2]$. The polynomial
$P$ is even in $x$, so no eigenvalue has absolute value larger than
$r_*$. Since $z=1$ gives $s=2$ and $r_*$ is a root of the first factor
in \eqref{eq:bloch-factor}, the value $r_*$ is attained. Hence
\[
\rhomax(A_{\sigma^{(8)}})=r_*.
\]
Finally, using $\pi<22/7$ and the alternating Taylor lower bound for cosine gives the
explicit estimates
\[
\cos(\pi/32)>0.9951,
\qquad
\cos(\pi/16)>0.9807.
\]
Therefore
\[
\rho_-(32)
>2\sqrt{0.9951^2+0.9807^2}
>2.7942>2.79361>r_*.
\]
Since both $\cos(\pi/n)$ and $\cos(2\pi/n)$ increase with $n$ for $n\ge32$,
$\rho_-(n)$ is increasing there. Therefore $r_*<\rho_-(n)$ for every $n=8m\ge32$.
\end{proof}

\begin{remark}[Additional periodic search]
We exhaustively scanned
the gauge-fixed subclass in which every step-$1$ edge has sign $+1$ and the step-$2$
signs have period $p\in\{2,4,8,16\}$. For $n=32$, the minimum over all $2^p$
patterns is respectively
$2.828427124746190$ for $p=2$, $2.828427124746190$ for $p=4$,
$2.793604493334841$ for $p=8$, and $2.793604493334841$ for $p=16$.
For $n=64$, the corresponding exhaustive minima for $p=8$ and $p=16$ are again

$2.793604493334841$. These numerical minima were reproduced by the supplied script
\texttt{verify\_periodic\_scan.py}, which evaluates the Hermitian
Fourier blocks for every one of the $2^p$ patterns. A pattern whose
fundamental period divides $p$ is included in that scan.
The step-$1$ Hamilton cycle has sign $+1$ in every
scanned signing. The opposite sign product gives a second sector that
switching cannot reach from this one. It supplies a strictly smaller
value, as the next theorem shows.
\end{remark}

\subsection{A smaller spectral radius with negative Hamilton-cycle sign}

\begin{theorem}[Antiperiodic competitor]\label{thm:antiperiodic}
For $n=8m$ with $m\ge4$, start with the period-$8$ signing of
Theorem~\ref{thm:counterexample}, number the vertices $0,\ldots,n-1$,
and reverse the signs on
\[
\{n-1,0\},\qquad \{n-2,0\},\qquad \{n-1,1\}.
\]
The resulting signed adjacency matrix $A_{\rm ap}$ satisfies
\[
\min_\sigma\rhomax(A_\sigma)
\le \rhomax(A_{\rm ap}) < r_* < \rho_-(n).
\]
\end{theorem}

\begin{proof}
The three edges listed above are exactly those crossing the boundary
between the last cell and the first. Thus the modified matrix acts as
the original period-$8$ operator on sequences extended by $v_{j+n}=-v_j$.
Write $v_{8k+a}=z^k h_a$ for cell $k$ and coordinate $0\le a<8$.
The boundary condition now gives $z^m=-1$. The $m$ distinct roots give
all $m$ cell Fourier modes, and the block on each mode is the same
matrix $H(z)$ as in Theorem~\ref{thm:counterexample}. Its parameter
$s=z+z^{-1}$ belongs to $[-2,2)$, since $z=1$ is excluded.

For $x\ge r_*$, the proof of Theorem~\ref{thm:counterexample} gave
$P(x,2)\ge0$ and $\partial P(x,s)/\partial s<0$ throughout $[-2,2]$.
Therefore $P(x,s)>P(x,2)\ge0$ for every admissible $s$. By evenness in
$x$, the determinant is positive also for $x\le-r_*$. Every eigenvalue
of every block lies in $(-r_*,r_*)$. There are finitely many blocks, so
$\rhomax(A_{\rm ap})<r_*$. The last inequality in the theorem is the
comparison already established in Theorem~\ref{thm:counterexample}.
\end{proof}

\emph{Independent finite certificate.} At $n=32$ direct
fraction-free elimination on the two integer matrices
$279I+100A_{\rm ap}$ and $279I-100A_{\rm ap}$ gives positive leading
principal minors of all $32$ orders for each matrix. By Sylvester's
criterion, $\rhomax(A_{\rm ap})<279/100$. Direct substitution gives
$f(279/100)=-6766519/10^8<0$, hence $279/100<r_*$. The accompanying
script \texttt{verify\_conjecture\_28.py} constructs the two matrices
from the signs above and checks all $64$ minors in integer arithmetic.
The unrestricted minimum on $C_{8m}(1,2)$ remains to be determined.

\begin{remark}[Computational status for the remaining small even sizes]
For $n=26,28,30$, in the heuristic searches described in the
accompanying computational repository, we found no signing with
spectral radius below $\rho_-(n)$. These searches are not exhaustive
proofs of optimality, and we make no claim that $32$ is the first even
value for which the global-optimality statement fails.
\end{remark}

\subsection{\texorpdfstring{The exceptional case $n=8$}{The exceptional case n=8}}

\begin{remark}[The exceptional case $n=8$]\label{rem:n8}
Identifying \eqref{eq:circ-system} with ``every quadrilateral of $C_n(1,2)$ is
unbalanced'' and hence with the parity-family construction of the present paper
requires $n\ge10$. At $n=8$ the step-$2$ edges form two additional
quadrilaterals, $S_0=(0,2,4,6)$ and $S_1=(1,3,5,7)$, which the
canonical signing of Proposition~\ref{prop:circ} balances. The class count and spectra from Propositions~\ref{prop:circ}
and~\ref{prop:twist} remain valid at $n=8$ for the system \eqref{eq:circ-system}
exactly as written: the proofs of Propositions~\ref{prop:circ}
and~\ref{prop:twist} apply verbatim (rank $7=n-1$, four classes,
$\rho\in\{2\sqrt2,\rho_-(8)\}$, the lattice maximum covered by
$g(\pi/8)>1$). The strictly larger system requiring all ten
quadrilaterals to be unbalanced is resolved by a flux computation.
Over $\F_2$,
\[
\chi_{S_0}=\chi_H+\!\!\sum_{i\ \mathrm{even}}\!\chi_{T_i},\qquad
\chi_{S_1}=\chi_H+\!\!\sum_{i\ \mathrm{odd}}\!\chi_{T_i},
\]
since each sum on the right covers every step-$1$ edge exactly once
and the four step-$2$ edges of the named quadrilateral exactly once.
On the solution family $\tau_i=\tau_0(-1)^i$, so
$\prod_{i\,\mathrm{even}}\tau_i=\prod_{i\,\mathrm{odd}}\tau_i=+1$ and
both extra quadrilaterals carry sign $\alpha$: the two additional
constraints are jointly equivalent to $\alpha=-1$. The all-ten system
therefore has exactly the two twisted classes $(\pm1,-1)$ as its
solutions, both attaining
\[
\rho_-(8)\;=\;2\sqrt{\cos^2(\pi/8)+\cos^2(\pi/4)}\;=\;\sqrt{4+\sqrt2}
\;=\;2.326846\ldots,
\]
and it excludes the $2\sqrt2$ class. The period-$8$ signing of
Theorem~\ref{thm:counterexample} is not a counterexample at $n=8$, since
its spectral radius is approximately $2.7936$, whereas
$\rho_-(8)\approx2.3268$.
\end{remark}

\begin{remark}
For even $n\ge10$, the spectral radius is constant on each of the four classes of the
solution family, so the family has only the two values
$\{\rho_-(n),\,2\sqrt2\}$ from Proposition~\ref{prop:twist}. Thus the
alternating triangle flux with anti-periodic step-$1$ holonomy attains
the minimum spectral radius within the quadrilateral-unbalanced family.
Theorem~\ref{thm:counterexample} gives an unrestricted signing with strictly
smaller spectral radius for $n=8m\ge32$. The constrained flux minimization
is nevertheless a useful exact finite model, analogous in spirit to
Lieb's flux-phase viewpoint for signed circulants \cite{lieb94}.

On $C_{480}(1,2)$ we sampled 100 uniform signings with seed 0.
The row labelled \texttt{circulant12} in the archived variance data
gives a mean spectral radius of $3.7141$, standard deviation $0.0669$,
and minimum $3.5285$. All 100 exceeded the $4$-regular tree benchmark
$2\sqrt{3}$; the largest radius in the constrained family is $2\sqrt2$.
This comparison concerns that graph and sample.

\end{remark}

\section{Outlook}

The circulant calculation gives three distinct benchmarks when
$n=8m\ge32$: the exact constrained minimum $\rho_-(n)$, the radius
$r_*$ of the positive-boundary period-eight construction, and the
strictly smaller radius of its antiperiodic counterpart. The last
quantity is an explicit upper bound on the unrestricted minimum.
Determining that minimum requires a lower bound valid for both
Hamilton-cycle sign products and all choices of triangle signs. An
initial question is whether the antiperiodic construction itself can
be improved at $n=32$.

For general graph families, Proposition~\ref{prop:master} leaves a
different question. When can the negative contributions from individual
short-cycle classes offset the excess of totally even walks while the
remaining classes stay controlled? The fixed-graph growth calculation
in Remark~\ref{rem:fixed-even} shows why a useful bound has to specify
how walk length varies with the graph. The bicycle-free conditioning
result gives a baseline against which a cancellation-sensitive estimate
could be compared.

\section*{Declaration of competing interest}
The author declares no competing interests.

\section*{Funding}
This research received no funding.

\section*{Data and code availability}
The \href{https://github.com/Vaibhavs25/bilu-linial-parity/tree/7416cbd12925558785194f90b555606c5f541d26}{archived code and finite data (commit \texttt{7416cbd})}
include the graph atlas, the variance measurements and their
generation scripts. From the repository root run
\texttt{python verify\_all.py} for the finite regression checks,
\texttt{python prism\_L.py} for the prism computation, and
\texttt{python trace\_rates.py} for the trace estimates. These three
scripts were rerun with NumPy 2.5.4, NetworkX 3.7 and SymPy 1.14.0.
The separate \texttt{verify\_conjecture\_28.py} supplied with this
article checks the $n=32$ certificate using integer arithmetic.

\begingroup
\setlength{\emergencystretch}{2em}
\raggedright

\endgroup


\small\begin{thebibliography}{20}
\bibitem{bass92} H.~Bass, \emph{The Ihara-Selberg zeta function of a
tree lattice}, Internat.~J.~Math.~3 (1992), 717-797.
\bibitem{bilulinial06} Y.~Bilu, N.~Linial,
\emph{Lifts, discrepancy and nearly optimal spectral gap},
Combinatorica 26 (2006), 495-519.
\bibitem{companion} V.~Suvagiya, \emph{Signed circulants at the Ramanujan
bound}, \href{https://arxiv.org/abs/2607.18334v1}{arXiv:2607.18334v1} (2026).
Withdrawn in version 2 upon consolidation into arXiv:2607.17343v2.
\bibitem{godsilgutman81} C.~D.~Godsil, I.~Gutman,
\emph{On the matching polynomial of a graph}, in: L.~Lov\'asz and V.~T.~S\'os (eds.),
\emph{Algebraic Methods in Graph Theory} (Proc., Szeged, 1978), Vol.~I,
pp.~241-249, Colloq.~Math.~Soc.~J\'anos~Bolyai, Vol.~25, North-Holland, Amsterdam, 1981.
\bibitem{heilmannlieb72} O.~J.~Heilmann, E.~H.~Lieb,
\emph{Theory of monomer-dimer systems}, Comm.~Math.~Phys.~25 (1972),
190-232.
\bibitem{huang19} H.~Huang, \emph{Induced subgraphs of hypercubes and a
proof of the Sensitivity Conjecture}, Ann.~of Math.~190 (2019), 949-955.
\bibitem{ks00} M.~Kotani, T.~Sunada, \emph{Zeta functions of finite
graphs}, J.~Math.~Sci.~Univ.~Tokyo 7 (2000), 7-25.
\bibitem{lieb94} E.~H.~Lieb, \emph{Flux phase of the half-filled band},
Phys.~Rev.~Lett.~73 (1994), 2158-2161.
\bibitem{mss15} A.~W.~Marcus, D.~A.~Spielman, N.~Srivastava,
\emph{Interlacing families I: Bipartite Ramanujan graphs of all degrees},
Ann.~of Math.~182 (2015), 307-325.
\bibitem{hornjohnson13} R.~A.~Horn, C.~R.~Johnson, \emph{Matrix Analysis}, 2nd ed.,
Cambridge University Press, 2013, Chapter~8.
\bibitem{mop20} S.~Mohanty, R.~O'Donnell, P.~Paredes,
\emph{Explicit near-Ramanujan graphs of every degree}, SIAM J.~Comput.~51
(2022), STOC20-1-STOC20-23; arXiv:1909.06988. DOI: 10.1137/20M1342112.
\bibitem{chenvandambu24} L.~Chen, E.~R.~van Dam, C.~Bu, \emph{Spectra of power
hypergraphs and signed graphs via parity-closed walks}, J. Comb. Theory Ser. A
207 (2024), 105909. DOI: 10.1016/j.jcta.2024.105909.
\bibitem{lihouwang24} D.~Li, Y.~Hou, D.~Wang, \emph{Zeta functions of signed
graphs}, Linear Multilinear Algebra 72 (2024), 2719-2731.
DOI: 10.1080/03081087.2023.2296578.
\bibitem{luoroy24} Y.~Luo, A.~Roy, \emph{Homological spectral graph theory and
weighted cycle counting}, \href{https://arxiv.org/abs/2403.01550v4}{arXiv:2403.01550v4} (2026).
\bibitem{prior26} V.~Suvagiya, \emph{Parity families and a kernel-averaged
$L$-function for near-Ramanujan signings},
\href{https://arxiv.org/abs/2607.17343v1}{arXiv:2607.17343v1} (2026).
\bibitem{consolidated26} V.~Suvagiya, \emph{Parity families and signed spectra:
kernel averaging, near-Ramanujan bounds, and exact circulant models},
\href{https://arxiv.org/abs/2607.17343v2}{arXiv:2607.17343v2} (2026).
\bibitem{rs21} M.~Ravichandran, N.~Srivastava, \emph{Asymptotically optimal
multi-paving}, Int. Math. Res. Not. IMRN 2021 (14) (2021), 10908--10940.
\href{https://doi.org/10.1093/imrn/rnz111}{doi:10.1093/imrn/rnz111}.
\bibitem{huang26} Z.~Huang, \emph{A $\sqrt2$-approximation to the Bilu--Linial
conjecture}, \href{https://arxiv.org/abs/2609.17551v2}{arXiv:2609.17551v2} (2026).
\bibitem{linzhou26} F.~Lin, H.~Zhou, \emph{Improved bounds for the Bilu--Linial
conjecture via spectral recovery from mixed determinantal polynomials},
\href{https://arxiv.org/abs/2609.15715v2}{arXiv:2609.15715v2} (2026).
\bibitem{xu26counter} Z.~Xu, \emph{A 3-regular counterexample to the Bilu--Linial
signing conjecture}, \href{https://arxiv.org/abs/2609.15591v4}{arXiv:2609.15591v4} (2026).
\end{thebibliography}
\end{document}